\documentclass{amsart}
\usepackage{amssymb}
\usepackage{amsthm}
\newtheorem{theor}{Theorem} 

\theoremstyle{definition} \newtheorem{defin}{Definition}
\newtheorem{ex}{Example}
\theoremstyle{remark} \newtheorem{rem}{Remark}
\begin{document}

\title{Enumeration of one class of plane weighted trees}
\author{Yury. Kochetkov}
\date{}
\email{yuyukochetkov@gmail.com}
\begin{abstract} By weighted tree we understand such connected
tree,that: a) each its vertex and each edge have a positive
integer weight; b) the weight of each vertex is equal to the sum
of weights of outgoing edges. Each tree has a binary structure ---
we can color its vertices in two colors, black and white so, that
adjacent vertices have different colors. A type is a set of
pairwise non-isotopic plane weighted trees with a given list of
weights of white vertices and a given list of weights of black
vertices. In this work we present a method for computing the
cardinality of a given type.
\end{abstract}
\maketitle

\section{Introduction}

We will consider plane connected trees with an additional
structure: edges and vertices have positive integer weights,
moreover, for each vertex the sum of weights of outgoing edges is
equal to the weight of this vertex. Such tree will be called a
\emph{w-tree}.

Each plane connected tree has a binary structure: its vertices can
be colored in black and white so, that adjacent vertices have
different colors. We will study the following problem: enumerate
pairwise non-isotopic w-trees with a given list of weights of
black vertices and a given list of weights of white vertices. The
set of such w-trees will be called a type. A type will be denoted
$\langle k_1,\ldots,k_s\,|\,l_1,\ldots,l_t\rangle$, where
$k_1,\ldots,k_s$ are weights of white vertices, written in non
decreasing order, and $l_1,\ldots,l_t$ are weights of black
vertices (also in non decreasing order). Obviously,
$k_1+\ldots+k_s=l_1+\ldots+l_t$. The number of vertices of the
type $\Xi$ will be denoted $v(\Xi)$. The cardinality of a type
$\Xi$ will be denoted $|\Xi|$.
\begin{ex} There are 6 w-trees in the type $\langle 5,12\,|\,
1,7,9\rangle$:
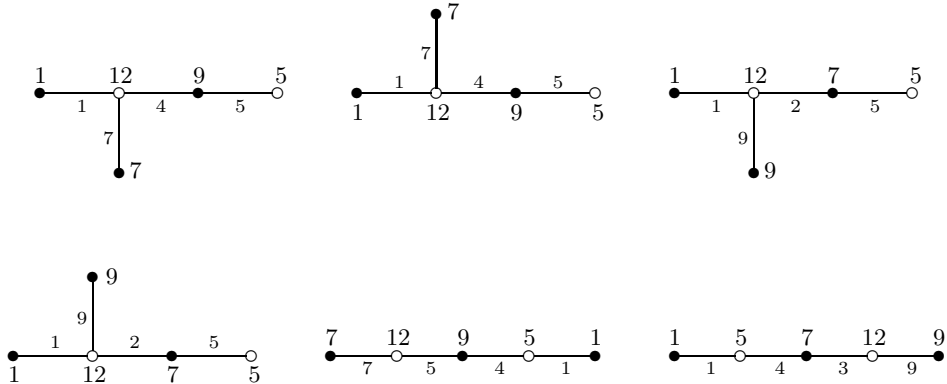
\[\begin{picture}(330,70) \put(0,35){\circle*{4}} \put(30,35){\circle{4}}
\put(30,5){\circle*{4}} \put(60,35){\circle*{4}} \put(90,35){\circle{4}}
\put(0,35){\line(1,0){28}} \put(30,5){\line(0,1){28}}
\put(32,35){\line(1,0){56}} \put(-2,39){\small 1} \put(26,39){\small 12}
\put(58,39){\small 9} \put(89,39){\small 5} \put(34,3){\small 7}
\put(14,28){\scriptsize 1} \put(44,28){\scriptsize 4} \put(74,28){\scriptsize
5} \put(24,16){\scriptsize 7}

\put(120,35){\circle*{4}} \put(150,35){\circle{4}} \put(150,65){\circle*{4}}
\put(180,35){\circle*{4}} \put(210,35){\circle{4}} \put(120,35){\line(1,0){28}}
\put(150,65){\line(0,-1){28}} \put(152,35){\line(1,0){56}} \put(118,25){\small
1} \put(146,25){\small 12} \put(178,25){\small 9} \put(209,25){\small 5}
\put(154,63){\small 7} \put(134,37){\scriptsize 1} \put(164,37){\scriptsize 4}
\put(194,37){\scriptsize 5} \put(144,48){\scriptsize 7}

\put(240,35){\circle*{4}} \put(270,35){\circle{4}} \put(270,5){\circle*{4}}
\put(300,35){\circle*{4}} \put(330,35){\circle{4}} \put(240,35){\line(1,0){28}}
\put(270,5){\line(0,1){28}} \put(272,35){\line(1,0){56}} \put(238,39){\small 1}
\put(266,39){\small 12} \put(298,39){\small 7} \put(329,39){\small 5}
\put(274,3){\small 9} \put(254,28){\scriptsize 1} \put(284,28){\scriptsize 2}
\put(314,28){\scriptsize 5} \put(264,16){\scriptsize 9} \end{picture}\]

\[\begin{picture}(350,50) \put(0,15){\circle*{4}} \put(30,15){\circle{4}}
\put(30,45){\circle*{4}} \put(60,15){\circle*{4}} \put(90,15){\circle{4}}
\put(0,15){\line(1,0){28}} \put(30,45){\line(0,-1){28}}
\put(32,15){\line(1,0){56}} \put(-2,5){\small 1} \put(26,5){\small 12}
\put(58,5){\small 7} \put(89,5){\small 5} \put(35,42){\small 9}
\put(14,18){\scriptsize 1} \put(44,18){\scriptsize 2} \put(74,18){\scriptsize
5} \put(24,27){\scriptsize 9}

\put(120,15){\circle*{4}} \put(145,15){\circle{4}} \put(170,15){\circle*{4}}
\put(195,15){\circle{4}} \put(220,15){\circle*{4}} \put(120,15){\line(1,0){23}}
\put(147,15){\line(1,0){46}} \put(197,15){\line(1,0){23}} \put(118,19){\small
7} \put(141,19){\small 12} \put(168,19){\small 9} \put(193,19){\small 5}
\put(218,19){\small 1} \put(132,8){\scriptsize 7} \put(156,8){\scriptsize 5}
\put(182,8){\scriptsize 4} \put(207,8){\scriptsize 1}

\put(250,15){\circle*{4}} \put(275,15){\circle{4}} \put(300,15){\circle*{4}}
\put(325,15){\circle{4}} \put(350,15){\circle*{4}} \put(250,15){\line(1,0){23}}
\put(277,15){\line(1,0){46}} \put(327,15){\line(1,0){23}} \put(248,19){\small
1} \put(273,19){\small 5} \put(298,19){\small 7} \put(321,19){\small 12}
\put(348,19){\small 9} \put(262,8){\scriptsize 1} \put(288,8){\scriptsize 4}
\put(312,8){\scriptsize 3} \put(338,8){\scriptsize 9} \end{picture}\] \end{ex}

\section{Plane graphs, Belyi functions and anti-Vandermonde systems}

To a w-tree we can correspond a tree-like plane graph with multiple edges :
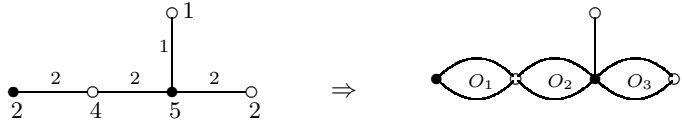
\[\begin{picture}(300,50) \put(0,15){\circle*{4}} \put(30,15){\circle{4}}
\put(60,15){\circle*{4}} \put(90,15){\circle{4}} \put(60,45){\circle{4}}
\put(0,15){\line(1,0){28}} \put(32,15){\line(1,0){56}}
\put(60,15){\line(0,1){28}} \put(-1,5){\small 2} \put(29,5){\small 4}
\put(59,5){\small 5} \put(89,5){\small 2} \put(64,43){\small 1}
\put(14,18){\scriptsize 2} \put(44,18){\scriptsize 2} \put(74,18){\scriptsize
2} \put(55,30){\scriptsize 1}

\put(120,13){$\Rightarrow$}

\put(160,20){\circle*{4}} \put(190,20){\circle{4}} \put(220,20){\circle*{4}}
\put(250,20){\circle{4}} \put(220,45){\circle{4}} \put(220,20){\line(0,1){23}}
\qbezier(160,20)(175,35)(189,21) \qbezier(160,20)(175,5)(189,19)
\qbezier(191,21)(205,35)(220,20) \qbezier(191,19)(205,5)(220,20)
\qbezier(220,20)(235,35)(249,21) \qbezier(220,20)(235,5)(249,19)
\put(172,17){\scriptsize $O_1$} \put(202,17){\scriptsize $O_2$}
\put(232,17){\scriptsize $O_3$} \end{picture}\] Here the weight of a vertex
becomes its valency and the weight of an edge becomes its multiplicity. In our
example the complement to the graph has 4 components: 3 are bounded (components
$O_1,O_2$ and $O_3$) and one unbounded.

Belyi function of this graph is a rational function $\varphi$ with critical
values $0,1,\infty$ such, that our graph and the graph $G=\varphi^{-1}[0,1]$
are isotopic. In our case $\varphi$ has 3 simple poles in points $c_1\in O_1$,
$c_2\in O_2$ and $c_3\in O_3$. White vertices of $G$ are inverse images of $0$
and black --- of $1$. Without loss of generality we can assume that point $0$
is the white vertex of valency 4 and point $1$ is the black vertex of valency
5. Let white vertex of valency 2 be at point $a_1$, white vertex of valency 1
--- at point $a_2$ and black vertex of valency 2 --- at point $b$. Thus,
$$\varphi(z)=\frac{\alpha z^4(z-a_1)^2(z-a_2)}{(z-c_1)(z-c_2)(z-c_3)},\quad
\varphi(z)-1=\frac{\alpha (z-1)^5(z-b)^2}{(z-c_1)(z-c_2)(z-c_3)}.$$ Hence,
$$\alpha z^4(z-a_1)^2(z-a_2)-(z-c_1)(z-c_2)(z-c_3)=\alpha (z-1)^5(z-b)^2.$$
Equaling coefficients at sixth, fifth and forth powers of $z$ we obtain a
system:
$$\left\{\begin{array}{l}2a_1+a_2=2b+5\\ a_1^2+2a_1a_2=b^2+10b+10\\
a_1^2a_2=5b^2+20b+10\end{array}\right.\eqno(1)$$ As
$$(2a_1+a_2)^2-2(a_1^2+2a_1a_2)=2a_1^2+a_2^2\text{ and } (2b+5)^2-2(b^2+10b+10)=
2b^2+5,$$ then the second equation can be written in the form
$2a_1^2+a_2^2=2b^2+5$. Analogously, as
$$3a_1^2a_2-3(2a_1+a_2)(a_1^2+2a_1a_2)+(2a_1+a_2)^3=2a_1^3+a_2^3$$ and
$$3(5b^2+20b+10)-3(2b+5)(b^2+10b+10)+(2b+5)^3=2b^3+5,$$  then the third equation
can be written in the form $2a_1^3+a_2^3=2b^3+5$. Thus, we obtain so called
\emph{anti-Vandermonde} \cite{Ko} system:
$$\left\{\begin{array}{l}2a_1+a_2-2b-5=0\\ 2a_1^2+a_2^2-2b^2-5=0\\
2a_1^3+a_2^3-2b^3-5=0\end{array}\right.$$ The number of finite solutions of
this system is $\leqslant 6$ (in our case it is 4).

In general case to a w-tree from the given type $\langle
k_1,k_2,\ldots,k_s\,|\, l_1,l_2,\ldots,l_t\rangle$ we correspond a plane
tree-like graph with $s$ white vertices of valences $k_1,k_2,\ldots,k_s$ and
$t$ black vertices of valences $l_1,l_2,\ldots,l_t$. Such graph has
$n=k_1+\ldots+k_s=l_1+\ldots+l_t$ edges and its complement is a union of
$m=n+1-s-t$ bounded domains and one unbounded domain. Belyi function of such
graph is a rational function $\varphi$ with simple poles at points
$c_1,\ldots,c_m$
--- one pole in each bounded domain. Let white vertices of the graph are at
points $a_1,\ldots,a_s$ and black --- at points $b_1,\ldots,b_t$.
So
$$\varphi(z)=\frac{\alpha (z-a_1)^{k_1}\cdots\ldots\cdot(z-a_s)^{k_s}}
{(z-c_1)\cdot\ldots\cdot(z-c_m)}= \frac{\alpha
(z-b_1)^{l_1}\cdots\ldots\cdot(z-b_t)^{l_t}}
{(z-c_1)\cdot\ldots\cdot(z-c_m)}+1.$$ Hence,
$$\alpha (z-a_1)^{k_1}\cdots\ldots\cdot(z-a_s)^{k_s}= \alpha
(z-b_1)^{l_1}\cdots\ldots\cdot(z-b_t)^{l_t}+ (z-c_1)\cdot\ldots\cdot(z-c_m).$$
Equaling coefficients at powers $z^{m+1},\ldots,z^{n-1}$ at both sides of
equation, we obtain a system with $n-1-m=s+t-2$ equations and $s+t$ unknowns
$a_1,\ldots,a_s,b_1,\ldots,b_t$. Without loss of generality, we can assume that
$a_s=0$ and $b_t=1$ (for example) and obtain a polynomial system of $s+t-2$
equations with $s+t-2$ unknowns that has $\leqslant (s+t-2)!$ solutions.

In next section we'll demonstrate that such system can always be transformed
into anti-Vandermonde system.

\section{Reduction to an anti-Vandermonde system}

Let
$$s=\prod_{i=1}^m (x-a_i)^{k_i}=x^n-s_1x^{n-1}+s_2x^{n-2}+\ldots,\, \text{ where }
n=\sum_{i=1}^m k_i.$$

\begin{theor} There exist polynomials $q_i$ with rational coefficients: $q_1$ from
one variable, $q_2$ from two variables and so on, such, that
$q_i(s_1,\ldots,s_i)=k_1a_1^i+\ldots+k_ma_m^i$. Moreover, coefficients of
polynomial $q_i$ depend only on $i$. \end{theor}

\begin{proof} Let us consider the product
$$s'=\prod_{i=1}^m (1-xa_i)^{k_i}=1-s_1x+s_2x^2+\ldots.$$ Then
\begin{align*}\ln(s')=\sum_{i=1}^m&k_i\ln(1-xa_i)=\\=&-\left(\sum_{i=1}^m
k_ia_i\right)x- \frac 12 \left(\sum_{i=1}^m
k_ia_i^2\right)x^2-\frac 13 \left(\sum_{i=1}^m k_ia_i^3\right)x^3-
\ldots.\hspace{10mm}(2)\end{align*} On the other hand,
$$\ln(s')=-(s_1x-s_2x^2+\ldots)-\frac 12 (s_1x-s_2x^2+\ldots)^2- \frac 13
(s_1x-s_2x^2+\ldots)^3-\ldots.\eqno(3)$$ Equaling coefficients at powers of $x$
in (2) and (3), we obtain the required polynomials. Thus,
$$\begin{array}{l} q_1=x_1,\,q_2=-2x_2+x_1^2,\,q_3=3x_3-3x_1x_2+x_1^3,\\
q_4=-4x_4+4x_1x_3+2x_2^2-4x_1^2x_2+x_1^4,\\
q_5:=5x_5-5x_1x_4-5x_2x_a3+5x_1^2x_3+5x_1x_2^2-5x_1^3x_2+x_1^5,\\
q_6=-6x_6+6x_1x_5+6x_2x_4+3x_3^2-6x_1^2x_4-\\\hspace{5cm}
-12x_1x_2x_3-2x_2^3+6x_1^3x_3 +9x_1^2x_2^2-6x_1^4x_2+x_1^6.
\end{array}$$ \end{proof}
\par\noindent
Let
$$\prod_{i=1}^s (x-a_i)^{k_i}=x^n-s_1x^{n-1}+s_2x^{n-2}+\ldots$$ and
$$\prod_{i=1}^t (x-b_i)^{l_i}=x^n-t_1x^{n-1}+t_2x^{n-2}+\ldots.$$ Then
$s_i=t_i$ for $i=1,\ldots,s+t-2$, and the system
$$s_1=t_1,\quad q_2(s_1,s_2)=q_2(t_1,t_2),\quad q_3(s_1,s_2,s_3)=
q_3(t_1,t_2,t_3),\quad\ldots$$ is the required ant-Vandermonde system.

\section{Simple, non-decomposable types}

\begin{defin} A type $\Xi=\langle k_1,k_2,\ldots,k_s\,|\, l_1,l_2,\ldots,l_t\rangle$ will
be called \emph{simple}, if numbers $k_1,k_2,\ldots,k_s$ are pairwise different
and numbers $l_1,l_2,\ldots,l_t$ also are pairwise different. $\Xi$ is called
\emph{decomposable}, if there exists a proper subset $I\subset\{1,\ldots,s\}$
and a proper subset $J\subset\{1,\ldots,t\}$ such, that
$$\sum_{i\in I} k_i=\sum_{j\in J} l_j.$$ In this case
$\langle\{k_i\}_{i\in I}\,|\,\{l_j\}_{j\in J}\rangle$ is also a type, which
will be called a \emph{subtype} of the type $\Xi$. \end{defin}

\begin{theor} Let $\langle k_1,k_2,\ldots,k_s\,|\,
l_1,l_2,\ldots,l_t\rangle$ be a simple non-decomposable type. Than
its cardinality is $(s+t-2)!$.
\end{theor}

\begin{rem} This theorem was proved by V.Dremov \cite{Dr} in his thesis in assumption that the
corresponding anti-Vandermonde system doesn't have any multiple solutions. \end{rem}

\begin{proof} Let us assume that the white vertex of the biggest weight $k_s$ is at
point $0$ and the black vertex of the biggest weight $l_t$ --- at $1$. The
corresponding anti-Vandermonde system is of the form:
$$\left\{\begin{array}{l} k_1x_1+\ldots+k_{s-1}x_{s-1}-l_1y_1-\ldots-
l_{t-1}y_{t-1}-l_t=0\\ \\ \qquad \cdots \qquad \cdots \qquad
\cdots \qquad\cdots\qquad\cdots\qquad\cdots\\ \\k_1x_1^{s+t-2}+
\ldots+k_{s-1}x_{s-1}^{s+t-2}-l_1y^{s+t-2}-\ldots-
l_{t-1}y_{s-1}^{l+s-2}-l_t=0\end{array}\right.$$ At first we'll
prove that the system has no solutions at infinity.
\par\medskip\noindent
A solution at infinity satisfies the system
$$\left\{\begin{array}{l} k_1x_1+\ldots+k_{s-1}x_{s-1}-l_1y_1-\ldots-
l_{t-1}y_{t-1}=0\\ \\ \qquad \cdots \qquad \cdots \qquad \cdots \qquad\cdots
\qquad\cdots\qquad\cdots\\
\\k_1x_1^{s+t-2}+ \ldots+k_{s-1}x_{s-1}^{s+t-2}-l_1y^{s+t-2}-\ldots-
l_{t-1}y_{s-1}^{l+s-2}=0\end{array}\right.$$ We can consider it as a
homogeneous system with unknowns $k_i,l_j$. As this system has a nonzero
solution, then the determinant of the system is zero, i.e.
$$\prod_{i=1}^{s-1}x_i\prod_{j=1}^{t-1}y_j\prod_{1\leqslant i<j\leqslant
t-1}(x_i-x_j)\prod_{1\leqslant i<j\leqslant
t-1}(y_i-y_j)\prod_{\scriptsize\begin{array}{l}1\leqslant i\leqslant s-1\\
1\leqslant j\leqslant t-1\end{array}}(x_i-y_j)=0.$$
\begin{itemize}
    \item If some $x_i=0$ (or $y_j=0$), then we eliminate terms $k_ix_i^r$ (or
    $l_jy_j^r$), $r=1,\ldots,s+t-2$, from the system.
    \item If $x_i=x_j$ (or $y_i=y_j$, or $x_i=y_j$), then we eliminate terms
    $k_jx_j^r$ (or $l_jy_j^r$) and coefficient $k_i$ ($l_i$) we replace by
    $k_i+k_j$ (by $l_i+l_j$, or by $k_i-l_j$).
\end{itemize} Also we eliminate the last equation.

We continue in the same way. As all variables $x_i$ and $y_j$ cannot be zeroes,
then at some step the coefficient at some variable $x_i$ or $y_j$ will become
zero. But this means that the type is decomposable.

To prove that our system has no multiple solutions let us consider a
anti-Vandermonde system
$$\begin{array}{l}k_1x_1+\ldots+k_nx_n+1=0\\ \qquad
\cdots \qquad \cdots \qquad \cdots \qquad\\
k_1x_1^n+\ldots+k_nx_n^n+1=0\end{array}$$ It defines a multivalued map from
$n$-dimensional space $K$ with coordinates $k_1,\ldots,k_n$ to $n$-dimensional
space $X$ with coordinates $x_1,\ldots,x_n$. The Jacobian of this map is
$$\prod_{i=1}^n x_i \prod_{1\leqslant i<j\leqslant n} (x_i-x_j)/n!
\prod_{i=1}^n k_i \prod_{1\leqslant i<j\leqslant n} (x_i-x_j).$$ We see that a
multiple solution can occur only when one of unknowns is zero,  or when values
of two unknowns coincide. In both cases we can diminish the number of unknowns
by one.

In the end we will obtain a system of the form
$$ly+1=ly^2+1=\ldots=ly^n+1=0,$$ that has a solution only if $l=-1$, but then our
type is decomposable.

We see that our system does not have solutions on infinity and does not have
multiple solutions. Hence, it has $(s+t-2)!$ finite pairwise different
solutions. \end{proof}

\section{Simple decomposable types}

If a type $\Xi$ is simple, but decomposable, then the the number
of its w-trees is less then $(v(\Xi)-2)!$.
\begin{ex} We have:
$$\begin{array}{ll}|\langle 1,5,7\,|\,2,4,7\rangle|=18, & |\langle
1,3,11\,|\,4,5,6\rangle|=20,\\
|\langle 1,2,4\,|\,1,2,4\rangle|=11, & |\langle 1,2,4,5\,|\,1,2,9\rangle|=72,\\
|\langle 1,2,3\,|\,1,2,3\rangle|=7, & |\langle 1,2,6,10\,|\,4,5,10\rangle|=96.
\end{array}$$
\end{ex}

\begin{defin} We say that a type $\Xi=\langle k_1,k_2,\ldots,k_s\,|\,
l_1,l_2,\ldots,l_t\rangle$ admits a $n$-partition if there exit
disjoint subsets $I_1,\ldots,I_n$ of the set $\{1,\ldots,s\}$, and
disjoint subsets $J_1,\ldots,J_n$ of the set $\{1,\ldots,t\}$ such
that
$$\bigcup_{p=1}^n I_p=\{1,\ldots,s\}, \quad \bigcup_{p=1}^n J_p=
\{1,\ldots,t\}$$ and
$$\sum_{i\in I_1}k_i=\sum_{j\in J_1}l_j,\quad\ldots,\quad \sum_{i\in I_n} k_i=
\sum_{j\in J_n}l_j.$$ We will say that the type $\Xi$ is a union of subtypes
$\Xi_1,\ldots,\Xi_n$: $\Xi=\Xi_1\cup\ldots\cup\Xi_n$, where
$$\Xi_1=\langle \{k_i\,|\,i\in I_1\}\,|\,\{l_j\,|\,j\in J_1\}\rangle,\ldots, \Xi_n=
\langle \{k_i\,|\,i\in I_n\}\,|\,\{l_j\,|\,j\in J_n\}\rangle.$$
\end{defin}

\begin{rem} The type itself can be considered as its 1-partition.\end{rem}

\begin{theor} The number of w-trees in a type $\Xi$ is a sum by
partitions of $\Xi$, where to each partition $\Xi=\Xi_1\cup\ldots\cup\Xi_n$,
$n=1,2,\ldots$, corresponds a summand
$$(-1)^{n-1}(v(\Xi)-1)^{n-2}\prod_{m=1}^n (v(\Xi_i)-1)!$$ \end{theor}

\begin{ex}
\par\noindent
\begin{itemize}
    \item The type $\langle 7,5,1\,|\,7,4,2\rangle$ has only one nontrivial partition
    (2-partition): $\langle 7,5,1\,|\,7,4,2\rangle
    =\langle 7\,|\,7\rangle\cup \langle 5,1\,|\,4,2\rangle$. Hence,
    $|\langle 7,5,1\,|\,7,4,2\rangle=4!-3!=18$.
    \item The type $\langle 4,2,1\,|\,4,2,1\rangle$ has following nontrivial
    partitions:
    \begin{enumerate}
        \item 2-partition $\langle 4\,|\,4\rangle\cup\langle
        2,1\,|\,2,1\rangle$;
        \item 2-partition $\langle 2\,|\,2\rangle\cup\langle
        4,1\,|\,4,1\rangle$;
        \item 2-partition $\langle 1\,|\,1\rangle\cup\langle
        4,2\,|\,4,2\rangle$;
        \item 3-partition $\langle 4\,|\,4\rangle\cup \langle
        2\,|\,2\rangle\cup \langle 1\,|\,1\rangle$.
    \end{enumerate} Hence, $|\langle 4,2,1\,|\,4,2,1\rangle|=
    4!-3!-3!-3!+5=11$.
    \item The type $\langle 3,2,1\,|\,3,2,1\rangle$ has following nontrivial
    partitions:
    \begin{enumerate}
        \item 2-partition $\langle 3\,|\,3\rangle\cup\langle
        2,1\,|\,2,1\rangle$;
        \item 2-partition $\langle 2\,|\,2\rangle\cup\langle
        3,1\,|\,3,1\rangle$;
        \item 2-partition $\langle 1\,|\,1\rangle\cup\langle
        3,2\,|\,3,2\rangle$;
        \item 2-partition $\langle 3\,|\,2,1\rangle\cup\langle
        2,1\,|\,3\rangle$;
        \item 3-partition $\langle 3\,|\,3\rangle\cup \langle
        2\,|\,2\rangle\cup \langle 1\,|\,1\rangle$.
    \end{enumerate} Hence, $|\langle 3,2,1\,|\,3,2,1\rangle|=
    4!-3!-3!-3!-2!\cdot 2!+5=7$.
\end{itemize} \end{ex}

\begin{rem} We can consider w-trees with non integral weights: the sum of
white weights must be equal to the sum of black. A solution of the
corresponding anti-Vandermonde system gives us positions of vertices.
\end{rem}

\begin{proof} \underline{The first step.} Let our type admits only one
partition, namely 2-partition $\Xi=\Xi'\cup\Xi''$, $v(\Xi')=k$, $v(\Xi'')=l$.

Let us define a "deformation" $\widetilde{\Xi}$ of our type $\Xi$: we increase
the weight of some white vertex $v$ from $\Xi'$ by small $\varepsilon$ and
increase the weight of some black vertex $u$ from $\Xi''$ by the same
$\varepsilon$. Thus obtained type is non-decomposable and contains
$(v(\Xi)-2)!$ w-trees. If $T$ is a w-tree from $\widetilde{\Xi}$, then weights
of its edges are positive integers, except edges of the path from $v$ to $u$.
The weight of such edges differs from integer on $\varepsilon$. If some edge of
this path has weight $\varepsilon$ (such edge will be called
$\varepsilon$-bridge), then such tree doesn't exist in the type $\Xi$
(connectivity is lost, when $\varepsilon$ becomes zero).

Thus, we need to find the number of w-trees in the type $\widetilde{\Xi}$ that
have $\varepsilon$-bridge. Let us note that $\varepsilon$-bridge connects a
white vertex $x$ from the subtype $\Xi'$ and a black vertex $y$ from the
subtype $\Xi''$: to a white vertex $x$ a path from $v$ comes with weight
$p-\varepsilon$ (where $p$ is an integer) and to a black vertex $y$ a path from
$u$ comes with weight $q-\varepsilon$. Hence, we can construct a tree with
$\varepsilon$-bridge in the following way: we choose a tree $T'$ from $\Xi'$
and a tree $T''$ from $\Xi''$ ($(k-2)!(l-2)!$ variants of such double choice);
we choose a white vertex of $T'$ and a direction of outgoing
$\varepsilon$-bridge (the number of such choices is equal to the number of $T'$
edges, i.e. $k-1$); we choose a black vertex from $T''$ and the direction of
ingoing $\varepsilon$-bridge ($l-1$ variants). Thus, $|\Xi|=(s+t-2)!-(k-1)!
(l-1)!$.
\par\medskip\noindent
\underline{Second step.} Let our type admits only 2-partitions and
$\Xi'_1\cup\Xi''_1,\ldots,\Xi'_n\cup\Xi''_n$ are all 2-partitions of the type
$\Xi$. Let $k_i$ be the number of vertices of subtype $\Xi'_i$ and $l_i$ be the
number of vertices of subtype $\Xi''_i$.

Let us increase the weight of a white vertex $x$ from $\Xi'_n$ by $\varepsilon$
and the weight of a black vertex $y$ from $\Xi''_n$ by the same $\varepsilon$.
Thus obtained deformed type $\widetilde{\Xi}$ can have only 2-partitions (the
number of such 2-partitions is strictly less, then $n$). Let us assume that
exactly first $k$ partitions of the type $\Xi$ are partitions of
$\widetilde{\Xi}$, i.e. vertices $x$ and $y$ both belong to $\Xi'_i$ or to
$\Xi''_i$, if $i\leqslant k$, and if $i>k$ then one of these two vertices
belongs to $\Xi'_i$ and another --- to $\Xi''_i$.

By induction
$$|\widetilde{\Xi}|=(s+t-2)!-\sum_{i=1}^k(k_i-1)!(l_i-1)!.$$ Now it remains to
find the number of trees in the type $\widetilde{\Xi}$ with
$\varepsilon$-bridge.

Let $T$ be such tree. If $\varepsilon=0$, then $T$ loses connectedness and is a
union of (exactly) two subtrees $T_1$ and $T_2$, where $T_1$ belongs to subtype
$\Xi'_i$ and $T_2$ --- to subtype $\Xi''_i$ for some $i>k$. As above we have
that the number of such trees $T$ is $(k_i-1)!(l_i-1)!$. And summation by $i$,
$k<i\leqslant n$, gives us the required formula.
\par\medskip\noindent
\underline{Third step.} Let now the type $\Xi$ has exactly one 3-partition
$\Xi=\Xi_1\cup\Xi_2\cup\Xi_3$ and the only 2-partitions of $\Xi$ are
$$\Xi_1\oplus(\Xi_2\cup\Xi_3),\,\,\Xi_2\oplus(\Xi_1\cup\Xi_3),\, \text{ and }
\Xi_3\oplus(\Xi_1\cup\Xi_2),$$ where by $\Xi'\oplus\Xi''$  we denote the type
whose set of white (black) weights is a union of set of white (black) weights
of type $\Xi'$ and set of white (black) weights of type $\Xi'$. Let subtypes
$\Xi_1$, $\Xi_2$ and $\Xi_3$ have $k$, $l$ and $m$ vertices respectively.

Let us increase weight of some white vertex from $\Xi_1$ by $\varepsilon$ and
weight of some black vertex from $\Xi_2$ by the same $\varepsilon$. Thus
obtained type $\widetilde{\Xi}$ has only partition --- 2-partition
$(\widetilde{\Xi_1}\oplus\widetilde{\Xi_2})\cup\Xi_3$. Hence,
$|\widetilde{\Xi}|=(k+l+m-2)!-(m-1)!(k+l-1)!$.

Now we must find the number of trees with $\varepsilon$-bridge. There can be
one bridge or two. In the first case either a tree from $\Xi_1$ is connected by
$\varepsilon$-bridge with a tree from $\Xi_2\oplus\Xi_3$, or a tree from
$\Xi_2$ is connected by $\varepsilon$-bridge with a tree from
$\Xi_1\oplus\Xi_3$. Hence, the number of trees with one $\varepsilon$-bridge is
\begin{multline*}(k-1)!(l+m-1)((l+m-2)!-(l-1)!(m-1)!)+\\
+(l-1)!(k+m-1)((k+m-2)!-(k-1)!(m-1)!).\end{multline*}

A tree with two $\varepsilon$-bridges can be constructed in the following way:
we choose a tree $T_1$ from $\Xi_1$, a tree $T_2$ from $\Xi_2$ and a tree $T_3$
from $\Xi_3$. A white vertex from $T_1$ we connect by $\varepsilon$-bridge with
a black vertex from $T_3$ and a white vertex from $T_3$ --- with a black vertex
from $T_2$. Hence the number of trees with two $\varepsilon$-bridges is
$(k-1)!(l-1)!(m-1)!(m-1)$. Now it is easy to check that we obtain the required
formula.

\par\medskip\noindent
\underline{Forth step.} Let the type has only 2- and 3-partitions. Then we
increase weight of some white vertex by $\varepsilon$ and weight some black
vertex by $\varepsilon$ with the aim to decrease the number of 3-partitions.
After that we use induction as in the second step.
\par\medskip\noindent
\underline{Fifth step.} Let type $\Xi$ has one 4-partition
$\Xi=\Xi_1\cup\Xi_2\cup\Xi_3\cup\Xi_4$ and all its 2- and 3-partitions are
constructed from subtypes $\Xi_i$. Let  the number of vertices in $\Xi_1$ be
$k$, in $\Xi_2$ --- $l$, in $\Xi_3$ --- $m$ and in $\Xi_4$ --- $n$.

As above let us define a deformation $\widetilde{\Xi}$ of the type $\Xi$, by
increasing the weight of some white vertex in $\Xi_1$ by $\varepsilon$ and some
black vertex in $\Xi_2$ by the same $\varepsilon$. After that we must find the
number of trees in $\widetilde{\Xi}$ with $\varepsilon$-bridges and subtract
this number from $|\widetilde{\Xi}|$. It is enough to check that thus obtained
expression contains the term $(k+l+m+n-1)^2(k-1)!(l-1)!(m-1)!(n-1)!$ (other
terms are related to 2- and 3-partitions). We will enumerate classes of trees
with bridges and for each class we will study with what coefficient the product
$(k-1)!(l-1)!(m-1)!(n-1)!$ appears in the formula for number of trees in this
class.
\begin{itemize}
    \item Class: $(\Xi_1\oplus\Xi_3\oplus\Xi_4)\cup\Xi_2$. One bridge.
    Coefficient: $(k+m+n-1)^2$.
    \item Class: $(\Xi_1\oplus\Xi_3)\cup(\Xi_2\oplus\Xi_4)$. One bridge.
    Coefficient: $(k+m-1)(l+n-1)$.
    \item Class: $(\Xi_1\oplus\Xi_4)\cup(\Xi_2\oplus\Xi_3)$. One bridge.
    Coefficient: $(k+n-1)(l+m-1)$.
    \item Class: $\Xi_1\cup(\Xi_2\oplus\Xi_3\oplus\Xi_4)$. One bridge.
    Coefficient: $(l+m+n-1)^2$.
    \item Class: $(\Xi_1\oplus\Xi_3)\cup\Xi_4\cup\Xi_2$. Two bridges.
    Coefficient: $-(k+m-1)(n-1)$.
    \item Class: $(\Xi_1\oplus\Xi_4)\cup\Xi_3\cup\Xi_2$. Two bridges.
    Coefficient: $-(k+n-1)(m-1)$.
    \item Class: $\Xi_1\cup(\Xi_3\oplus\Xi_4)\cup\Xi_2$. Two bridges.
    Coefficient: $-(m+n-1)^2$.
    \item Class: $\Xi_1\cup\Xi_3\cup(\Xi_2\oplus\Xi_4)$. Two bridges.
    Coefficient: $-(l+n-1)(m-1)$.
    \item Class: $\Xi_1\cup\Xi_4\cup(\Xi_2\oplus\Xi_3)$. Two bridges.
    Coefficient: $-(l+m-1)(n-1)$.
    \item Class: $\Xi_1\cup\Xi_3\cup\Xi_4\cup\Xi_2$. Three bridges.
    Coefficient: $2(m-1)(n-1)$.
\end{itemize}
It is easy to find the sum of coefficients
\begin{multline*} (k+m+n-1)^2+(k+m-1)(l+n-1)+(k+n-1)(l+m-1)+(l+m+n-1)^2-\\
-(k+m-1)(n-1)-(k+n-1)(m-1)-(m+n-1)^2-(l+m-1)(n-1)-\\-(l+n-1)(m-1)+2(m-1)(n-1)
=(k+l+m+n-1)^2.\end{multline*} There is a geometrical interpretation of this
combinatorial identity.

Let us consider a square $A=[0,k+l+m+n-1]\times
[0,k+l+m+n-1]\subset\mathbb{R}^2$ and four rectangles in it:
$$\begin{array}{l}A_1=[0,k+m+n-1]\times [0,k+m+n-1],\\
A_2=[k,k+l+m+n-1]\times[k,k+l+m+n-1],\\
A_3=[0,k+m-1]\times[k+m,k+l+m+n-1],\\ A_4=[k+n,k+l+m+n-1]\times[0,k+n-1].
\end{array}$$  It is easy to see that $A=A_1\cup A_2\cup A_3\cup A_4$ and
$$\begin{array}{l} A_1\cap\ A_2=[k,k+m+n-1]\times[k,k+m+n-1],\\ A_1\cap
A_3=[0,k+m-1]\times[k+m,k+m+n-1],\\ A_1\cap A_4=[k+n,k+m+n-1]\times
[0,k+n-1],\\ A_2\cap A_3=[k,k+m-1]\times[k+m,k+l+m+n-1],\\ A_2\cap A_4=
[k+n,k+l+m+n-1]\times[k,k+n-1],\\ A_3\cap A_4=\varnothing,\\ A_1\cap A_2\cap
A_3=[k,k+m-1]\times[k+m,k+m+n-1],\\ A_1\cap A_2\cap A_4=[k+n,k+m+n-1]\times
[k,k+n-1].\end{array}$$ Thus, our identity is simply the inclusion-exclusion
formula.
\par\medskip\noindent
\underline{Sixth step.} In the general case we consider a partition of the type
$\Xi$, analogous to the partition in the fifth step, but (n+2)-partition.
Analogously we define deformation $\widetilde{\Xi}$. We must find the
coefficient at product $\prod_{i=1}^{n+2}(m_i-1)!$ in the formula, that
describe the number of trees with bridges (here $m_i=v(\Xi_i)$,
$i=1,\ldots,n+2$). Let $M=m_1+\ldots+m_{n+2}-1$.

As above we will enumerate classes of trees with bridges and for
each class we will find the corresponding coefficient. A class is
defined by such partition (admissible partition) of the set
$\{1,\ldots,n+2\}$ into disjoint subsets, that numbers 1 and 2
belong to \emph{different} subsets. To a given admissible
partition we correspond a product by subsets of this partition and
then we compute a sum by all admissible partitions.

To a subset $\{s_1,\ldots,s_p\}\subset\{1,\ldots,n+2\}$ of given partition we
correspond the factor
$$(m_{s_1}+\ldots+m_{s_p}-1)^p,\text{ if }\{1,2\}\cap\{s_1,\ldots,s_p\}=
\varnothing,$$ and the factor
$$(m_{s_1}+\ldots+m_{s_p}-1)^{p-1},\text{ if
} s_1=1 \text{ or } s_1=2.$$ Thus obtained product by subsets of the partition
we multiply by $(-1)^{q-2}(q-2)!$, where $q$ is the number of subsets in the
partition.

We must prove, that the sum by all admissible partitions is equal
to $M^n$. The proof will use the inclusion-exclusion formula.

Let us consider an $n$ dimensional cube $C$ and right parallelopes in it. The
cube $C$ is  the Cartesian product: $C=[0,M]^n$. A right parallelotope is a
Cartesian product of $n$ segments. Parallelotopes under consideration are
enumerated by subsets (proper and non-proper) of the set $N=\{3,\ldots,n+2\}$.
Given the set $A=\{i_1,\ldots,i_s\}\subset\{3,\ldots,n+2\}$ we define segments
$[a_i,b_i]$, $i=1,\ldots,n$,
$$a_i=\begin{cases}\begin{array}{ll}0,&\text{ if } i+2\in A,\\ m_1+S_A,&\text{ if }
i+2\notin A;\end{array}
\end{cases}\quad b_i=\begin{cases}\begin{array}{ll}m_1+S_A-1,&\text{ if }
i+2\in A,\\ M,&\text{ if } i+2\notin A.\end{array}\end{cases}$$ and
parallelotope $P_A=[a_1,b_1]\times\ldots\times[a_n,b_n]\subset C$ (here
$S_A=\sum_{i\in A}m_i$) with volume
$$\text{volume}(P_A)=\left(m_1+S_A-1\right)^{|A|}\left(M-m_1-S_A\right)^{n-|A|}.$$
We must prove that $\bigcup_A P_A=C$. Let $X=(x_1,\ldots,x_n)$ be a point in
$C$. If $x_i\leqslant m_1+m_{i+2}-1$, $i=1,\ldots,n$, then $X\in P_N$. If
$x_i\geqslant m_1+m_{i+2}$, $i=1,\ldots,n$, then $X\in P_\varnothing$. Let now
$$x_i\leqslant m_1+m_{i+2}-1,\,i=1,\ldots,s\text{ and } x_i\geqslant m_1+m_{i+2},\,
i=s+1,\ldots,n$$ and let us denote by $F$ the set $\{3,\ldots,s+2\}$. If
$x_i\leqslant m_1+S_F-1, i=s+1,\ldots,n$, then $X\in P_N$. If $x_i\geqslant
m_1+S_F, i=s+1,\ldots,n$, then $X\in P_F$. If $m_1+S_F-1\leqslant
x_{i-2}\leqslant m_1+S_F$, for $i\in G\subset \{s+3,\ldots,n+2\}$, then instead
of $F$ we will consider the set $F\cup G$. And so on.

Parallelotops $P_A$ and $P_B$ have a non-empty intersection only
if $A\subset B$ (or $B\subset A$). In this case
$$V(P_A\cap
P_B)=(k+S_B-1)^{|S_B|}(S_A-S_B-1)^{|S_A)-|S_B|}(M-k-S_A)^{n-|S_A|}.$$ Triple
intersections $P_{A\cup B\cup C}\cap P_{A\cup B}\cap P_A$ and $P_{A\cup B\cup
C}\cap P_{A\cup C}\cap P_A$, where $A,B,C\subset\{3,\ldots,n+2\}$ are pairwise
disjoint subsets, have the same volume. Analogously, for each non-empty
intersection $P_A\cap P_B\cap P_C\cap P_D$ there are five more with the same
volume.
\end{proof}

\section{Non-simple types}

If a type $\Xi$ is not simple, then a numeration of vertices (or a
labeling of vertices) makes the type simple. Thus obtained type
will be called a "derivative" type and will be denoted $\Xi'$.

\begin{ex} If $\Xi=\langle 6,1,1,1\,|\,3,3,3\rangle$, then
$\Xi'=\langle 6,1,1',1''\,|\,3,3',3''\rangle$. This type is simple, but
decomposable ---  it admits three 2-partitions:
\begin{enumerate}
\item $\langle 6\,|\,3,3'\rangle\cup\langle 1,1',1''\,|\,3''\rangle$; \item
$\langle 6\,|\,3,3''\rangle\cup\langle 1,1',1''\,|\,3'\rangle$; \item $\langle
6\,|\,3',3''\rangle\cup\langle 1,1',1''\,|\,3\rangle$.
\end{enumerate} Hence, there are 84
w-trees in this type. 72 of them are of the form
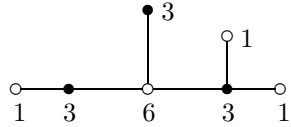
\[\begin{picture}(100,50) \put(0,15){\circle{4}} \put(20,15){\circle*{4}}
\put(50,15){\circle{4}} \put(50,45){\circle*{4}} \put(80,15){\circle*{4}}
\put(80,35){\circle{4}} \put(100,15){\circle{4}} \put(2,15){\line(1,0){46}}
\put(52,15){\line(1,0){46}} \put(50,17){\line(0,1){28}}
\put(80,15){\line(0,1){18}} \put(-1,3){1} \put(18,3){3} \put(48,3){6}
\put(78,3){3} \put(99,3){1} \put(55,41){3} \put(85,31){1}
\end{picture}\] and 12 are symmetric with the third order symmetry
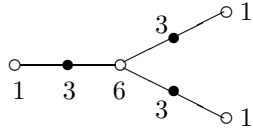
\[\begin{picture}(90,50) \put(0,25){\circle{4}} \put(40,25){\circle{4}} \put(80,5){\circle{4}}
\put(80,45){\circle{4}} \put(20,25){\circle*{4}} \put(60,15){\circle*{4}}
\put(60,35){\circle*{4}} \put(2,25){\line(1,0){36}} \put(41,26){\line(2,1){38}}
\put(41,24){\line(2,-1){38}} \put(-1,12){1} \put(37,12){6} \put(18,12){3}
\put(53,5){3} \put(53,37){3} \put(85,2){1} \put(85,42){1}
\end{picture}\] Thus, the number of w-trees in the type $\langle
6,1,1,1\,|\,3,3,3\rangle$ is
$$\frac{72}{3!\,3!}+\frac{12}{\frac{3!\,3!}{3}}=3$$ --- the correct answer.
\end{ex}
Now we can formulate the statement.

\begin{theor} Let $\Xi$ be a non simple type and let $\{k_1,\ldots,k_s\}$ be the
set of pairwise different white weights and $\{l_1,\ldots,l_t\}$ be the set of
pairwise different black weights. Let there be $m_i$ vertices with weight
$k_i$, $i=1,\ldots,s$ and $n_j$ vertices with weight $l_j$, $j=1,\ldots,t$. Put
$p=\prod_{i=1}^s (m_i)!\prod_{i=1}^t(n_i)!$. If the derivative type $\Xi'$
contains $N$ nonsymmetric w-trees and $M_i$, $i=2,3,\ldots$, w-trees with
$i$-order symmetry, then
$$|\Xi|=\frac Np+\sum_i \frac{iM_i}{p}.$$ \end{theor}

\begin{rem} A type $\Xi$ contains a w-tree with $i$-order symmetry, if
\begin{enumerate}
    \item there exists a vertex $v$ with weight divisible by $i$;
    \item in the set of all other vertices for each color and each weight the
    number of vertices of these weight and color is divisible by $i$.
    and
\end{enumerate}

Such type generates the type $\Xi/i$, where the weight of the vertex $v$ is in
$i$ times smaller and the number of vertices of given color and weight is in
$i$ times smaller, then in the type $\Xi$. Thus, for the type $\Xi=\langle
6,1,1,1\,|\,3,3,3\rangle$ the type $\Xi/3=\langle 2,1\,|\,3\rangle$. Each
w-tree of type $\Xi/i$ generates the unique w-tree of the type $\Xi$. \end{rem}

\begin{ex} Let us consider the type $\Xi=\langle 5,2,1,1,\,|\,3,3,3\rangle$ and
its derivative type $\Xi'=\langle 5,2,1,1',\,|\,3,3',3''\rangle$. The
derivative type admits six 2-partitions, hence, $|\Xi'|=120-6\cdot 12=48$. As
there are no symmetric w-trees, then $|\Xi|=4$. Four w-trees from the type
$\Xi$ are presented below:
\[\begin{picture}(270,50) \put(0,15){\circle{4}} \put(30,15){\circle*{4}}
\put(60,15){\circle{4}} \put(90,15){\circle*{4}} \put(120,15){\circle{4}}
\put(60,45){\circle*{4}} \put(90,45){\circle{4}} \put(2,15){\line(1,0){56}}
\put(62,15){\line(1,0){56}} \put(60,17){\line(0,1){28}} \put(90,15){\line
(0,1){28}} \put(-2,5){\small 2} \put(28,5){\small 3} \put(58,5){\small 5}
\put(88,5){\small 3} \put(119,5){\small 1} \put(64,43){\small 3}
\put(94,43){\small 1}

\put(150,15){\circle{4}} \put(180,15){\circle*{4}} \put(210,15){\circle{4}}
\put(240,15){\circle*{4}} \put(270,15){\circle{4}} \put(180,45){\circle{4}}
\put(210,45){\circle*{4}} \put(152,15){\line(1,0){56}}
\put(212,15){\line(1,0){56}} \put(210,17){\line(0,1){28}} \put(180,15){\line
(0,1){28}} \put(149,5){\small 1} \put(178,5){\small 3} \put(208,5){\small 5}
\put(238,5){\small 3} \put(268,5){\small 2} \put(184,43){\small 1}
\put(214,43){\small 3} \end{picture}\]

\[\begin{picture}(280,50) \put(0,25){\circle{4}} \put(30,25){\circle*{4}}
\put(60,25){\circle{4}} \put(80,35){\circle*{4}} \put(80,15){\circle*{4}}
\put(100,45){\circle{4}} \put(100,5){\circle{4}} \put(2,25){\line(1,0){56}}
\put(61,26){\line(2,1){38}} \put(61,24){\line(2,-1){38}} \put(-2,15){\small 2}
\put(28,15){\small 3} \put(58,15){\small 5} \put(105,2){\small 1}
\put(105,42){\small 1} \put(73,6){\small 3} \put(73,37){\small 3}

\put(130,25){\circle*{4}} \put(160,25){\circle{4}} \put(190,25){\circle*{4}}
\put(220,25){\circle{4}} \put(250,25){\circle*{4}} \put(270,10){\circle{4}}
\put(270,40){\circle{4}} \put(130,25){\line(1,0){28}}
\put(162,25){\line(1,0){56}} \put(222,25){\line(1,0){28}}
\put(250,25){\line(4,3){19}} \put(250,25){\line(4,-3){19}} \put(128,15){\small
3} \put(158,15){\small 5} \put(188,15){\small 3} \put(218,15){\small 2}
\put(248,15){\small 3} \put(275,7){\small 1} \put(275,37){\small 1}
\end{picture}\] \end{ex}

\end{document}